\documentclass{amsart}  

\usepackage{amsrefs}
\usepackage[usenames, dvipsnames]{color}
\usepackage{amsfonts}
\usepackage{amsthm}
\usepackage{amsmath}
\usepackage{amsfonts}
\usepackage{latexsym}
\usepackage{amssymb}
\usepackage{amscd}
\usepackage[latin1]{inputenc}
\usepackage{verbatim}
\usepackage{enumerate}
\usepackage{comment}
\usepackage{mathrsfs}

\newtheorem{theorem}{Theorem}[section]
\newtheorem{lemma}[theorem]{Lemma}
\newtheorem{proposition}[theorem]{Proposition}
\newtheorem{corollary}[theorem]{Corollary}

\newtheorem{definition}[theorem]{Definition}

\theoremstyle{definition}

\newcommand\tr{ \mbox{Tr} }

\newcommand\diag{\operatorname{diag}}

\allowdisplaybreaks

\begin{document}

\title{L$^p$ spaces of operator-valued functions}

\author[Christopher Ramsey]{Christopher Ramsey}
\address{Department of Mathematics and Statistics, MacEwan University, Edmonton, AB, Canada}
\email{ramseyc5@macewan.ca}

\author{Adam Reeves}
\address{School of Mathematics and Statistics, Carleton University, Ottawa, ON, Canada}
\email{adamreeves@cmail.carleton.ca}


\keywords{Positive operator-valued measure, POVM, Quantum random variable, Operator-valued integration, Schatten norm, Decomposable norm}
\subjclass[2010]{ 
46G10, 
 47L05, 
 81P15 
 }


\maketitle

\begin{abstract}    
We define a $p$-norm in the context of quantum random variables, measurable operator-valued functions with respect to a positive operator-valued measure. This norm leads to a operator-valued $L^p$ space that is shown to be complete. Various other norm candidates are considered as well as generalizations of H\"older's inequality to this new context.
\end{abstract}

\section{Introduction} 

In recent works \cite{PR1,PR2}, the first author and Sarah Plosker have defined the spaces of $L^\infty$ and $L^1$ quantum random variables with respect to an operator-valued measure, meaning measurable operator-valued functions on a nice set. These Banach spaces are suitably well-behaved to look into the possibility of $L^p$ spaces in this context. This paper will discuss these new spaces and the difficulties of trying to establish versions of H\"older's and triangle (Minkowski's) inequalities in this context. One of the main difficulties is the (mis)behaviour of the Loewner (positive) order. Sections 2 and 3 introduce the $p$-norms and their decomposable versions as well as other possibilities for norms in these contexts. Section 4 proves some H\"older inequality analogues but all with extra assumptions.

First we remind the reader of the main objects of study. Suppose that $\mathcal H$ is a finite dimensional or separable Hilbert space, $X$ is a locally compact Hausdorff space and $\mathcal O(X)$ is the $\sigma$-algebra of Borel sets on $X$.

An {\em operator-valued measure (OVM)} $\nu: \mathcal O(X) \rightarrow \mathcal B(\mathcal H)$ is an ultraweakly countably additive function. That is, for every countable collection $\{E_k\}_{k\in\mathbb N} \subset \mathcal O(X)$ of disjoint Borel sets one has
\[
\nu\left(\bigcup_{k\in\mathbb N} E_k\right) = \sum_{k\in\mathbb N} \nu(E_k),
\]
with the sum converging in the ultraweak topology. To define this topology one first needs to define the set of {\em states} or {\em density operators} $\mathcal S(\mathcal H)$, that is the set of all positive, trace-one operators in $\mathcal B(\mathcal H)$. Thus, to every $s\in\mathcal S(\mathcal H)$ one has that $\tr(s\: \cdot)$ is a state on $\mathcal B(\mathcal H)$ in the other sense, a unital positive linear functional. Recall, that there are many more states in this latter sense that do not arise from $\mathcal S(\mathcal H)$. However, the density operators do separate points in $\mathcal B(\mathcal H)$ and norm any positive operator. With all of this in mind, $A_n \rightarrow A$ ultraweakly in $\mathcal B(\mathcal H)$ if and only if 
\[
\tr(sA_n) \rightarrow \tr(sA), \quad \forall s\in\mathcal S(\mathcal H).
\]

An OVM $\nu: \mathcal O(X)\rightarrow \mathcal B(\mathcal H)$ is called a positive, operator-valued measure (POVM) if $\nu(E) \geq 0$ for every $E\in\mathcal O(X).$ Notice that this implies that $\nu$ is a finite measure since $\nu(E) \leq \nu(X) \in \mathcal B(\mathcal H)$, $E\in\mathcal O(X)$. Such POVMs arise as one of the main objects of study in quantum physics and quantum information theory in particular. See \cite{BuschOp} for a canonical presentation of the probabilistic structure of quantum mechanics. In the operator theoretic context these objects have been studied in \cite{HLLL, FPS, FK, FKP, FFP, MPR, PR1} with a lot of development in the past decade due to the growing connections between pure mathematics and quantum theory.

For any full-rank density operator $\rho\in \mathcal S(\mathcal H)$, that is, $\rho$ is a positive, trace-one operator with no null space, define $\nu_\rho(E) := \tr(\rho \nu(E)), E\in \mathcal O(X)$. Then $\nu_\rho$ is a finite, positive measure into $\mathbb C$ that is mutually absolutely continuous with $\nu$. Note that the definition of absolute continuity of measures extends easily to all OVMs as it does not depend on the dimension of the codomain.

Suppose now that $\mathcal H$ has an orthonormal basis $\{e_i\}$, finite or countable. Define $\nu_{ij}(E) := \langle \nu(E)e_j,e_i\rangle, E\in\mathcal O(X)$ which is a finite measure such that $\nu_{ij} \ll_{ac} \nu \sim_{ac} \nu_\rho$. Hence, by the Radon-Nikod\'ym Theorem there exists a unique $\frac{d\nu_{ij}}{d\nu_\rho} \in L^1(X,\nu_\rho)$ such that $\nu_{ij}(E) = \int_E \frac{d\nu_{ij}}{d\nu_\rho} d\nu_\rho$ for all $E\in \mathcal O(X).$ Finally, define the {\em Radon-Nikod\'ym derivative of $\nu$} as 
\[
\frac{d\nu}{d\nu_\rho} := \left[ \frac{d\nu_{ij}}{d\nu_\rho} \right]_{i,j}.
\]
We are only interested in POVMs whose Radon-Nikod\'ym derivative gives a bounded operator a.e., which happens automatically in finite-dimensions.
It was shown in \cite{MPR} that this is the same under any choice of full-rank $\rho$ or orthonormal basis. Moreover, if $\nu$ is a POVM then $\frac{d\nu}{d\nu_\rho}$ maps into the positive operators.

A {\em quantum random variable} $f:X \rightarrow \mathcal B(\mathcal H)$ is a Borel measurable function, that is, for each $s\in \mathcal S(\mathcal H)$ the function $\tr(sf(x))$ is Borel measurable.
Now for each state $s\in \mathcal S(\mathcal H)$ define
\[
f_s(x) := \tr\left(s\left(\frac{d\nu}{d\nu_\rho}(x)\right)^{1/2}f(x)\left(\frac{d\nu}{d\nu_\rho}(x)\right)^{1/2}\right).
\]
A quantum random variable $f:X\rightarrow \mathcal B(\mathcal H)$ is {\em $\nu$-integrable} if and only if $f_s$ is $\nu_\rho$-integrable for every $s\in \mathcal S(\mathcal H)$. Finally, the integral of $f$ with respect to $\nu$ is implicitly defined by
\[
\tr\left( s\int_X f\: d\nu \right) =\int_X f_s \: d\nu_\rho, \quad s\in\mathcal S(\mathcal H).
\]
It is helpful to realize that if $\mu$ is a finite, positive measure on $X$ and $\nu = \mu I_\mathcal H$ then $\frac{d\nu}{d\nu_\rho} = I_\mathcal H$ and integration is just entrywise integration. Refer to \cite{MPR, PR1, PR2} for further reading on quantum random variable integration.

\section{The $L^p$ norm}

Here and in the remainder of the paper we will be assuming that $\nu : \mathcal O(X) \rightarrow \mathcal B(\mathcal H)$ is a POVM where the Radon-Nikod\'ym derivative $\frac{d\nu}{d\nu_\rho}$ is a quantum random variable, maps into bounded operators. 

As mentioned in the introduction, the only $L^p$ spaces of quantum random variables defined previously are those for $p=1$ and $\infty$. First, in \cite{PR1} the von Neumann algebra of all essentially bounded quantum random variables,
\[
L^\infty_\mathcal H(X,\nu) = L^\infty(X,\nu_\rho) \ \bar\otimes \ \mathcal B(\mathcal H)
\]
was introduced. Here, the $\infty$-norm is defined to be the least upper essential bound, 
$\|f\|_\infty \geq \|f(x)\|$ almost everywhere with respect to $\nu$ (or equivalently $\nu_\rho$).
Second, in \cite{PR2} an $L^1$ norm and Banach space were defined which will be given below in the more general $p$-norm definition.
For
\[
\mathcal L_\mathcal H(X,\nu) = \operatorname{span}\{\textrm{positive quantum random variables}\}
\]
we define:

\begin{definition}
If $f\in \mathcal L_\mathcal H(X,\nu)$ consider the set of 4-tuples of positive quantum random variables that combine to give $f$:
\[
Pos_f := \Big\{(f_1,f_2,f_3,f_4)\in \mathcal L_\mathcal H(X,\nu)^4 : f = f_1 - f_2 + if_3 -if_4, f_i\geq 0, 1\leq i\leq 4\Big\}
\]
Now define
\begin{align*}
\|f\|_{p} & = \inf_{Pos_f} \sup_{s\in\mathcal S(\mathcal H)} \left(\int_X (f_1 + f_2 + f_3+ f_4)_s^p d\nu_\rho\right)^{1/p}
\\ &= \inf_{Pos_f} \sup_{s\in\mathcal S(\mathcal H)} \|(f_1 +f_2 + f_3+f_4)_s\|_{p}
\end{align*}
where the latter $p$-norm is that of $L^p(X,\nu_\rho)$.
\end{definition}

When $p=1$ this is indeed the 1-norm defined in \cite{PR2} since
\begin{align*}
\inf_{Pos_f} \sup_{s\in\mathcal S(\mathcal H)} & \left(\int_X \tr(s(f_1 + f_2 + f_3+ f_4)) \: d\nu_\rho \right) 
\\ & \quad \quad = \inf_{Pos_f}\sup_{s\in\mathcal S(\mathcal H)} \tr\left(s\int_X f_1+f_2+f_3+f_4\:  d\nu\right)
\\ &\quad \quad= \inf_{Pos_f} \left\|\int_X f_1+f_2+f_3+f_4 \: d\nu \right\|
\\ &\quad \quad= \|f\|_{1}.
\end{align*}

\begin{proposition}
$\|\cdot\|_{p}$ is a seminorm on $\mathcal L_\mathcal H(X,\nu)$ into the extended reals $[0,\infty]$.
\end{proposition}
\begin{proof}
Positive semidefinite is automatic. For homogeneity, if $\lambda = \lambda_1 - \lambda_2 + i\lambda_3 - i\lambda_4 \in \mathbb C$ with $\lambda_i\geq 0$ and $(f_1,f_2,f_3,f_4)\in Pos_f$ then we have that the sum of the positive parts of 
$\lambda f  = (\lambda_1 - \lambda_2 +i\lambda_3 - i\lambda_4)(f_1-f_2 +if_3 - if_4)$
is
\begin{align*}
(\lambda_1 + \lambda_2 + \lambda_3+\lambda_4)(f_1+f_2+f_3+f_4) & = (|Re \lambda| + |Im \lambda|)(f_1 + f_2 + f_3 + f_4)
\\ & = |\lambda|_1 (f_1 + f_2 + f_3 + f_4).
\end{align*}
Thus, $\|\lambda f\|_{p} \leq |\lambda|_1\|f\|_{p}$. But this is an equality since if $(g_1, g_2, g_3, g_4) \in Pos_{\lambda f}$ then $f = \lambda^{-1}(g_1 - g_2 + ig_3 -ig_4)$ which can be recombined as the linear combination of positive functions and so all possibilities in the infimum are realized. Hence, $\|\lambda f\|_{p} = |\lambda|_1\|f\|_{p}$ and $\|\cdot\|_{p}$ is homogeneous with respect to the 1-norm on $\mathbb C$.

Now for the triangle inequality, let $f,g\in \mathcal L_\mathcal H(X,\nu)$ with $(f_1,f_2,f_3,f_4)\in Pos_f$ and $(g_1,g_2, g_3,g_4)\in Pos_g$. By the triangle inequality on $L^p(X,\nu_\rho)$ we have
\begin{align*}
\|f+g\|_{p} & \leq \sup_{s\in\mathcal S(\mathcal H)} \left\|\left(\sum_{i=1}^4 f_i + g_i\right)_s\right\|_{p}
\\ &\leq \sup_{s\in\mathcal S(\mathcal H)} \left\|\left(\sum_{i=1}^4 f_i\right)_s\right\|_{p} +  \left\|\left(\sum_{i=1}^4 g_i\right)_s\right\|_{p}
\\ & \leq \sup_{s\in\mathcal S(\mathcal H)} \left\|\left(\sum_{i=1}^4 f_i\right)_s\right\|_{p} + \sup_{s\in\mathcal S(\mathcal H)} \left\|\left(\sum_{i=1}^4 g_i\right)_s\right\|_{p}.
\end{align*}
Taking the infimum over the sets $Pos_f$ and $Pos_g$ results in
\[
\|f+g\|_{p} \leq \|f\|_{p} + \|g\|_{p}
\]
\end{proof}

Let $\mathcal L^p_\mathcal H(X,\nu) = \{f\in \mathcal L_\mathcal H(X,\nu) : \|f\|_{p} < \infty\}$ and $I_p = \{f\in \mathcal L^p_\mathcal H(X,\nu) : \|f\|_{p} = 0\}$. Naturally then, define
\[
L^p_\mathcal H(X,\nu) := \mathcal L^p_\mathcal H(X,\nu)/I_p.
\]
The very nice thing about this $p$-norm is that it allows us a completeness argument.

\begin{theorem}[cf. Theorem 3.12 \cite{PR2}] \label{thm:LpBanach}
$L^p_\mathcal H(X,\nu)$ is a Banach space with respect to the $\|\cdot\|_{p}$ norm.
\end{theorem}
\begin{proof}
The proof is exactly that in \cite{PR2} with minor changes to this new wider context.
Let $\{f_n\}$ be a Cauchy sequence in $L^p_\mathcal H(X,\nu)$.
This implies that there are numbers $\{k_n\}_{n\in \mathbb N}$ such that 
\[
\|f_l - f_m\|_{p} < \frac{1}{2^{n+1}}, \quad \forall l,m \geq k_n.
\]
Since $f_{k_1} \in L^p_\mathcal H(X,\nu)$ there exist $(g_{0,1},g_{0,2},g_{0,3},g_{0,4})\in Pos_{f_{k_1}}$ such that
\[
\left\| g_{0,1} + g_{0,2} + g_{0,3} + g_{0,4}\right\|_{p} < \|f_{k_1}\|_{p} + 1.
\]
Similarly, $f_{k_{n+1}} - f_{k_n} \in L^p_\mathcal H(X,\nu)$ there exists $(g_{n,1},g_{n,2},g_{n,3},g_{n,4})\in Pos_{f_{k_{n+1}} - f_{k_n}}$ such that
\[
\left\| g_{n,1} + g_{n,2} + g_{n,3} + g_{n,4}\right\|_{p} < \|f_{k_{n+1}} - f_{k_n}\|_{p} + \frac{1}{2^{n+1}} < \frac{1}{2^n}.
\]
Now for each $s\in\mathcal S(\mathcal H)$ the sequence $\left(\sum_{n=0}^m g_{n,1} + g_{n,2} + g_{n,3} + g_{n,4}\right)_s$ is increasing almost everywhere on $X$ and is bounded above
\begin{align*}
 \left(\int_X \left(\sum_{n=0}^m g_{n,1} + g_{n,2} + g_{n,3} + g_{n,4}\right)_s^p d\nu_\rho\right)^{1/p}  & \leq \sum_{n=0}^m \left\| g_{n,1} + g_{n,2} + g_{n,3} + g_{n,4}\right\|_{p} 
 \\ & < \|f_{k_1}\|_{p} + \sum_{n=0}^m \frac{1}{2^n} 
 \\ & < \|f_{k_1}\|_{p} + 2. 
\end{align*}
Thus by the Monotone Convergence Theorem, $\sum_{n=0}^\infty g_{n,1} + g_{n,2} + g_{n,3} + g_{n,4} \in L^p_\mathcal H(X,\nu)$ with
\[
\left\| \sum_{n=0}^\infty g_{n,1} + g_{n,2} + g_{n,3} + g_{n,4} \right\|_{p} <  \|f_{k_1}\|_{p} + 2.
\]
An identical argument gives that
\[
\left\| \sum_{n=0}^\infty g_{n,i} \right\|_{p} \leq \left\| \sum_{n=0}^\infty g_{n,1} + g_{n,2} + g_{n,3} + g_{n,4} \right\|_{p}
\]
and so $g_i := \sum_{n=0}^\infty g_{n,i}\geq 0$ is in $L^p_\mathcal H(X,\nu)$. Thus, $g := g_1 - g_2 + ig_3 - ig_4 \in L^p_\mathcal H(X,\nu)$.
Consider now that for each $m\geq 1$, by telescoping, we have that
\begin{align*}
\|g - f_{k_m}\|_{p} & = \left\|f_{k_1} + \sum_{n=1}^\infty (f_{k_{n+1}} - f_{k_n}) - f_{k_m}\right\|_{p}
\\ & = \left\|f_{k_m} + \sum_{n=m}^\infty (f_{k_{n+1}} - f_{k_n}) - f_{k_m}\right\|_{p}
\\ & = \left\| \sum_{n=m}^\infty (f_{k_{n+1}} - f_{k_n}) \right\|_{p}
\\ & < \sum_{n=m}^\infty \frac{1}{2^n}
\\ & = \frac{1}{2^{m-1}}.
\end{align*}
Therefore, $f_{n} \rightarrow g$ with respect to $\|\cdot\|_{p}$ and the conclusion follows.
\end{proof}

There are certainly other possibilities for defining a $p$-norm. One would be to ignore the operator structure of $f$ and look at 
\[
\left\| \int_X \|f(x)\|^p I_\mathcal H d\nu \right\|^{1/p}.
\]
As described in \cite{PR2}, this type of norm is not very useful in an infinite setting. For instance, $f(x) = \sum_{i=n}^\infty 2^{n}E_{n,n} \chi_{(1/2^{n}, 1/2^{n-1}]}$ on $X = [0,1]$ with $\nu = \mu I_{\ell^2(\mathbb N)}$, where $\mu$ is Lebesgue measure, is infinite in the above norm.

Another possibility would be to consider
\begin{align*}
\inf_{Pos_f} \left\|\left(\sum_{i=1}^4 f_i\right)^p \right\|_1^{1/p}& = \inf_{Pos_f}  \left\| \int_X \left(\sum_{i=1}^4 f_i\right)^p \: d\nu \right\|^{1/p} 
\\ & = \inf_{Pos_f} \sup_{s\in\mathcal S(\mathcal H)} \left(\int_X \tr\left( s\frac{d\nu}{d\nu_\rho}^{1/2}\left(\sum_{i=1}^4 f_i\right)^p\frac{d\nu}{d\nu_\rho}^{1/2} \right) d\nu_\rho\right)^{1/p}
\end{align*}
as a $p$-norm candidate. This is natural to consider as this is the 1-norm when $p=1$. 
Unfortunately, this probably fails the triangle inequality outside of the cases $p=1$ and $2$. We say ``probably'' since this doesn't seem to be known.

\subsection{Schatten-type norms}

A final class of potential norms to consider are those arising from a combination of Schatten norms and classical integral norms.

Recall, that in $\mathcal B(\mathcal H)$, $\|A\|_{S^p} = \tr(|A|^p)^{1/p}$ is called the Schatten $p$-norm. When $p=1$ this is called the trace norm and when $p=2$ it is called the Hilbert-Schmidt norm. The operators which have a finite $p$-norm are called the Schatten class for $p$ and are non-closed (in the operator norm) ideals of the compact operators. These are norms because there are Young's, H\"older's and Minkowski's inequalities in this context. See \cite{Dell} for further background on these norms and their classes of operators. 

Consider now this new family of seminorms:
\begin{align*}
\|f\|_{S^p, L^q} & = \sup_{s\in \mathcal S(\mathcal H)}\left(\int_X \tr\left(\left|s^{1/2}\frac{d\nu}{d\nu_\rho}^{1/2}f\frac{d\nu}{d\nu_\rho}^{1/2}s^{1/2}\right|^p\right)^{q/p} d\nu_\rho\right)^{1/q}
\\ & =  \sup_{s\in \mathcal S(\mathcal H)} \left\| \left\| s^{1/2}\frac{d\nu}{d\nu_\rho}^{1/2}f\frac{d\nu}{d\nu_\rho}^{1/2}s^{1/2} \right\|_{S^p} \right\|_{L^q}.
\end{align*}

\begin{lemma}
$\|\cdot\|_{S^p, L^q}$ is a seminorm on $\mathcal L_\mathcal H(X,\nu)$ with $\|f\|_{S^p, L^q} \leq \|f\|_{S^r, L^q}$, for $p\geq r$. Furthermore, $\|f\|_{S^1,L^q} \leq \|f\|_q$.
\end{lemma}
\begin{proof}
It is immediate that for any fixed $s\in \mathcal S(\mathcal H)$ that
\[
f \mapsto \left\| \left\| s^{1/2}\left(\frac{d\nu}{d\nu_\rho}(x)\right)^{1/2}f(x)\left(\frac{d\nu}{d\nu_\rho}(x)\right)^{1/2}s^{1/2} \right\|_{S^p} \right\|_{L^q}\] is a seminorm. Thus, $\|\cdot\|_{S^p, L^q}$ is the supremum of seminorms and so is a seminorm itself.

Monotonicity follows since the Schatten norms have this monotonic property.

For the last inequality, suppose $f\in\mathcal L_\mathcal H(X,\nu)$. For any $(f_1,f_2,f_3,f_4)\in Pos_f$ we have
\begin{align*}
\|f\|_{S^1, L^q} & = \sup_{s\in\mathcal S(\mathcal H)} \left\| \left\|s^{1/2}\frac{d\nu}{d\nu_\rho}^{1/2}(f_1-f_2 + if_3 - if_4)\frac{d\nu}{d\nu_\rho}^{1/2}s^{1/2} \right\|_{S^1} \right\|_{L^q}
\\ & \leq \sup_{s\in\mathcal S(\mathcal H)} \left\| \sum_{i=1}^4 \left\|s^{1/2}\frac{d\nu}{d\nu_\rho}^{1/2}f_i\frac{d\nu}{d\nu_\rho}^{1/2}s^{1/2} \right\|_{S^1} \right\|_{L^q}
\\ & = \sup_{s\in\mathcal S(\mathcal H)} \left\|\tr\left(s^{1/2}\frac{d\nu}{d\nu_\rho}^{1/2}(f_1+f_2+f_3+f_4)\frac{d\nu}{d\nu_\rho}^{1/2}s^{1/2} \right) \right\|_{L^q}
\\ & = \|f_1 +f_2+f_3+f_4\|_{q}.
\end{align*}
Taking the infimum over all such positive decompositions yields the desired result.
\end{proof}

One of the main difficulties with this seminorm is that there is no clear path to showing that it would lead to a Banach space (at least to the authors). A possiblity would be to show that $\|\cdot\|_p$ and $\|\cdot\|_{S^1,L^p}$ are comparable, though this may prove to not be true.

\section{The decomposable $p$-norm}

In 1985 Haagerup \cite{Haag} introduced the decomposition (or decomposable) norm for the completely bounded maps that are in the span of the completely positive maps. In particular, for such a map $u : \mathcal A\rightarrow \mathcal B$ he defined
\[
\|u\|_{dec} = \inf_{S_1,S_2}\{\max\{\|S_1\|, \|S_2\|\}\}
\]
where $S_1$ and $S_2$ are completely positive maps such that 
\[
a \mapsto \left[\begin{matrix} S_1(a) & u(a) \\ u(a^*)^* & S_2(a) \end{matrix}\right]
\]
is a completely positive map. For further reference see Chapter 6 of Pisier's book \cite{Pisier}. Junge and Ruan \cite{JungeRuan} used this to define the decomposable norm of their non-commutative L$^p$ space, $L^p(\mathcal M)$ for a von Neumann algebra $\mathcal M$. The original norm that this decomposable norm is formed from is the Schatten $p$-norm. 

In the same way, we can use these ideas to create a decombosable norm in our context. 

\begin{definition}
 If $f\in \mathcal L_\mathcal H(X,\nu)$ then define 
\begin{align*}
\|f\|_{p, dec} = \inf\Bigg\{ \max\{ \|S_1\|_{p}, \|S_2\|_{p} \} : S_1,S_2\in \mathcal L_\mathcal H(X,\nu), S_1,S_2\geq 0,
\\ \left[\begin{matrix} S_1(x) & f(x) \\ f(x)^* & S_2(x) \end{matrix}\right] \geq 0 \ \textrm{a.e.} \Bigg\}
\end{align*}
\end{definition}

Proving that this is a seminorm doesn't depend on the norm one starts with. For completeness we provide the proof but one can equally find something similar in \cite{Haag} or \cite{Pisier}.

\begin{proposition}
$\|\cdot\|_{p, dec}$ is a seminorm on $\mathcal L_\mathcal H(X,\nu)$.
\end{proposition}
\begin{proof}
Positive semidefinite is automatic. For $\lambda\in \mathbb C, \lambda\neq 0$ let $\lambda^{1/2}$ be the principle square root and then
\[
\diag(\lambda^{1/2},\overline{\lambda^{1/2}})\left[\begin{matrix} S_1(x) & f(x) \\ f(x)^* & S_2(x) \end{matrix}\right]\diag(\overline{\lambda^{1/2}},\lambda^{1/2}) = \left[\begin{matrix} |\lambda|S_1(x) & \lambda f(x) \\ (\lambda f(x))^* & |\lambda|S_2(x) \end{matrix}\right].
\]
Thus, it is immediate that $\|\lambda f\|_{p, dec} = |\lambda|\|f\|_{p, dec}$. This gives that the dec-norm is homogeneous with respect to the usual Euclidean norm instead of the 1-norm on $\mathbb C$, certainly an improvement on $\|\cdot\|_{p}$.

Finally, if $f,g\in \mathcal L_\mathcal H(X,\nu)$ and $S_1,S_2,T_1,T_2\in \mathcal L_\mathcal H(X,\nu)$ are positive such that 
\[
\left[\begin{matrix} S_1 & f \\ f^* & S_2 \end{matrix}\right], \ \left[\begin{matrix} T_1 & g \\ g^* & T_2 \end{matrix}\right] \geq 0
\]
then 
\[
\left[\begin{matrix} S_1+ T_1 & f+g \\ (f+g)^* & S_2+T_2 \end{matrix}\right] \geq 0
\]
and
\[
\max\{ \|S_1 + T_1\|_{p}, \|S_2 + T_2\|_{p}\} \leq \max\{\|S_1\|_{p},\|S_2\|_{p}\} + \max\{\|T_1\|_{p},\|T_2\|_{p}\}.
\]
Taking infimums we get that $\|f+g\|_{p, dec} \leq \|f\|_{p, dec} + \|g\|_{p, dec}$.
\end{proof}

We now establish that the ``nice'' properties of the $\|\cdot\|_p$ norm are inherited by the decomposable $p$-norm by way of comparability.

\begin{proposition}
If $f\in\mathcal L_\mathcal H(X,\nu)$ then $\|f^*\|_{p, dec} = \|f\|_{p, dec}$ and if
$f = f^*$ then $\|f\|_{p, dec}= \|f\|_{p}$. 
\end{proposition}
\begin{proof}
These arguments follow very similarly to \cite[Chapter 6]{Pisier}.
The first statement follows from the fact that
\[
\left[\begin{matrix} S_2 & f^* \\ f & S_1 \end{matrix}\right] = \left[\begin{matrix} 0 & I \\ I & 0 \end{matrix}\right]\left[\begin{matrix} S_1 & f \\ f^* & S_2 \end{matrix}\right]\left[\begin{matrix} 0 & I \\ I & 0 \end{matrix}\right].
\]
Now suppose that $f=f^*$. For any combination $f=f_1 - f_2$ such that $f_1,f_2\in \mathcal L_\mathcal H(X,\nu)$ are positive we have 
\[
\left[\begin{matrix} f_1 + f_2 & f \\ f & f_1 + f_2 \end{matrix}\right] = \left[\begin{matrix} f_1  & f_1 \\ f_1 & f_1 \end{matrix}\right] + \left[\begin{matrix} f_2 & -f_2 \\ -f_2 & f_2 \end{matrix}\right].
\]
Thus, $\|f\|_{p, dec} \leq \|f_1+f_2\|_{p}$ and taking infimums gives one direction.

For the other direction assume that $\left[\begin{matrix} S_2(x) & f(x) \\ f(x) & S_1(x) \end{matrix}\right] \geq 0$ in $M_2(\mathcal B(\mathcal H))$ a.e. By \cite[Lemma 1.37]{Pisier} 
\[
|\langle f\xi,\xi\rangle| \leq \frac{1}{2}\langle S_1\xi,\xi\rangle + \frac{1}{2}\langle S_2\xi,\xi\rangle, \quad \forall \xi\in \mathcal H
\]
which implies that
\[
\left\langle \left(\frac{1}{2}(S_1 + S_2) \pm f\right)\xi,\xi \right\rangle\geq 0, \quad \forall \xi\in\mathcal H.
\]
Thus,
\begin{align*}
\|f\|_{p} & = \left\| \frac{1}{2}\left(\frac{1}{2}(S_1 + S_2) + f\right) - \frac{1}{2}\left(\frac{1}{2}(S_1 + S_2) - f\right) \right\|_{p}
\\ & \leq \left\|\frac{1}{2}\left(\frac{1}{2}(S_1 + S_2) + f\right) + \frac{1}{2}\left(\frac{1}{2}(S_1 + S_2) - f\right) \right\|_{p}
\\ & = \left\|\frac{1}{2}(S_1 + S_2)\right\|_{p}
\\ & \leq \max\{\|S_1\|_{p},\|S_2\|_{p}\}
\end{align*}
Therefore, $\|f\|_{p} \leq \|f\|_{p, dec}$.
\end{proof}

\begin{proposition}
For $f\in\mathcal L_\mathcal H(X,\nu)$, $\|Re\:f\|_{p},\|Im\:f\|_{p}\leq \|f\|_{p}$ and $\|Re\:f\|_{p,dec}$, $\|Im\:f\|_{p,dec}\leq \|f\|_{p,dec}$. Moreover, 
\[
\frac{1}{2}\|f\|_{p} \leq \|f\|_{p,dec} \leq 2\|f\|_{p}.
\]
\end{proposition}
\begin{proof}
Suppose $(f_1,f_2,f_3,f_4) \in Pos_f$. Then
\[
\|Re\:f\|_{p} \leq \|f_1 + f_2\|_{p} \leq \|f_1+f_2+f_3+f_4\|_{p}
\]
since $(f_1+f_2)_s \leq (f_1 + f_2 + f_3 + f_4)_s$ for all $s\in \mathcal S(\mathcal H)$.
Taking the infimum over $Pos_f$ we get  $\|Re\: f\|_{p} \leq \|f\|_{p}$. The imaginary part follows similarly.

For the dec-norm suppose that $S_1,S_2\in\mathcal L_\mathcal H(X,\nu), S_1,S_2\geq 0$ with $\left[\begin{matrix} S_1 & f \\ f^* & S_2 \end{matrix}\right]\geq 0$. This implies
\[
0\leq \frac{1}{2}\left[\begin{matrix} S_1 & f \\ f^* & S_2 \end{matrix}\right] + \frac{1}{2}\left[\begin{matrix} S_2 & f^* \\ f & S_1 \end{matrix}\right] = \left[\begin{matrix} \frac{1}{2}(S_1+S_2) & Re\:f \\ Re\:f & \frac{1}{2}(S_1+S_2) \end{matrix}\right]
\]
and so
\[
\|Re\:f\|_{p,dec} \leq \frac{1}{2}\|S_1+S_2\|_{p} \leq \max\{\|S_1\|_{p},\|S_2\|_{p}\}.
\]
Taking the infimum gives $\|Re\:f\|_{p,dec} \leq \|f\|_{p,dec}$. Again, the imaginary part follows similarly.

Lastly, the comparability of the norms follows immediately by the triangle inequality, the agreement of the two norms on self-adjoint terms, and the arguments above.
\end{proof}

\begin{corollary}
$L^p_\mathcal H(X,\nu)$ is a Banach space under the $\|\cdot\|_{p,dec}$ norm. 
\end{corollary}
\begin{proof}
This is a direct consequence of the comparability of the norms and Theorem \ref{thm:LpBanach}.
\end{proof}

Finally, one can estimate the decomposable $p$-norm by the operator absolute value of the function and its adjoint. 

\begin{lemma}\label{lemma:absolutevalue}
Suppose $f\in\mathcal L_\mathcal H(X,\nu)$ has polar decompositions $f = u|f| = |f^*|u$ with partial isometry $u\in \mathcal L_\mathcal H(X,\nu)$. Then
\[
\|f\|_{p,dec} \leq \max\Big\{ \| |f| \|_{p}, \||f^*|\|_{p}\Big\}.
\]
\end{lemma}
\begin{proof}
Observe
\[
\left[\begin{matrix} |f^*| & f \\ f^* & |f|\end{matrix}\right] = \left[\begin{matrix} u & 0 \\ 0 & I\end{matrix}\right]\left[\begin{matrix} |f| & |f| \\ |f| & |f|\end{matrix}\right]\left[\begin{matrix} u^* & 0 \\ 0 & I\end{matrix}\right] \geq 0.
\]
Hence, $\|f\|_{p,dec} \leq \max\{\||f|\|_{p}, \||f^*|\|_{p}\}$.
\end{proof}

\section{H\"{o}lder's inequality}

In the classical $L^p$ context one proves Young's inequality to prove H\"older's inequality to prove the triangle inequality (or Minkowski's inequality). Here we have already arrived at the triangle inequality in the previous sections. This is by design of course as H\"older's inequality does not hold in the generality one would wish. However, much can be said and this section explores the possibilities.

\begin{theorem}\label{HI2}
Suppose $\mathcal H = \mathbb C^n$ and $\frac{1}{q}+\frac{1}{p}=1$.
If $f,g \in \mathcal L_\mathcal H(X,\nu)$ such that $f^p, g^q\in L^1_\mathcal H(X,\nu)$ $f,g\geq 0$ and $fg=gf$, then  $fg\in L^1_\mathcal H(X,\nu)$ and 
\[
\|fg\|_1 
\leq
\|f^p\|_1^{1/p}\|g^q\|_1^{1/q}.
\]
\end{theorem}

\begin{proof}


Since $f$ and $g$ commute then $fg\geq 0$. 
By Young's inequality for commuting matrices developed by Ando \cite{MYI} we get:
\begin{align*}
    \frac{f(x) g(x) }{A  B}
    \leq
    \frac{f(x)^p}{A^p p}+\frac{g(x)^q}{B^q q}
\end{align*}
Then by the comparison theorem
\begin{align*}
    \bigg{\|}\int_X\frac{f(x) g(x)}{A B}d\nu\bigg{\|}
    &\leq
    \bigg{\|}\int_X\frac{f(x)^p}{A^p p}+\frac{g(x)^q}{B^q q}d\nu\bigg{\|}\\
    &\leq
    \bigg{\|}\int_X\frac{f(x)^p}{A^p p}d\nu\bigg{\|}+\bigg{\|}\int_X\frac{g(x)^q}{B^q q}d\nu\bigg{\|}
\end{align*}

Hence, letting $A=\|f^p\|_1^{1/p}$ and $B=\|g^q\|_1^{1/q}$ leads to
\begin{align*}
\left\|\int_X\frac{f(x)g(x)}{A B}d\nu\right\|
&\leq
\frac{1}{A^p p}\left\|\int_Xf(x)^pd\nu \right\|+\frac{1}{B^q q}\left\|\int_Xg(x)^qd\nu\right\|\\
&=
\frac{1}{p}+\frac{1}{q}\\
&=
1
\end{align*}
Therefore,
\begin{align*}
\|fg\|_1 & = \left\|\int_Xf g d\nu\right\|
\\ & \leq
AB\\
& = \|f^p\|_1^{1/p}\|g^q\|_1^{1/q}
\end{align*}
\end{proof}

\begin{corollary}
Suppose $\mathcal H = \mathbb C^n$ and $\frac{1}{q}+\frac{1}{p}=1$.
If $f,g \in \mathcal L_\mathcal H(X,\nu)$ such that $\|f(x)\|^pI_n, g^q\in L^1_\mathcal H(X,\nu)$ and $f,g\geq 0$  then $g^\frac{1}{2}fg^\frac{1}{2} \in L^1_\mathcal H(X,\nu)$ and 
\begin{align*}
\|g^\frac{1}{2} f g^\frac{1}{2}\|_1
\leq
\big\|\|f(x)\|^pI_n\big\|_1^{1/p} \|g^q\|_1^{1/q}
\end{align*}
\end{corollary}

\begin{proof}
By positivity
\[
f(x) \leq \|f(x)\| I_n
\]
and so
\[
g(x)^\frac{1}{2} f(x) g(x)^\frac{1}{2}
\leq
\|f(x)\|g(x)
\]
Now since $g(x)$ and $\|f(x)\| I_n$ commute, we can apply the previous theorem to get
\[
\|g^\frac{1}{2} f g^\frac{1}{2}\|_1
\leq
\big\| \|f(x)\|g(x)\big\|_1
\leq
\big\|\|f(x)\|^p I_n\big\|_1^{1/p} \|g^q\|_1^{1/q}.
\]
\end{proof}

We can use the previous theorem to establish a limited triangle inequality.

\begin{theorem}
Suppose $\mathcal H=\mathbb C^n$. If $f,g\in \mathcal L_\mathcal H(X,\nu)$ such that $f^p,g^p\in L^1_\mathcal H(X,\nu)$, $f,g\geq 0$ and $fg = gf$ then $(f+g)^p\in L^1_\mathcal H(X,\nu)$ and 
\[
\|(f+g)^p\|_1^{1/p} \leq \|f^p\|_1^{1/p} + \|g^p\|_1^{1/p}.
\]
\end{theorem}
\begin{proof}
The classic proof of the triangle (or Minkowski) inequality still works in this case. Namely, because $f$ and $g$ are commuting positive operators then by the previous theorem and assuming $\frac{1}{p} + \frac{1}{q} = 1$ 
\begin{align*}
\|(f+g)^p)\|_1 & = \|(f+g)(f+g)^{p-1}\|_1
\\ & \leq \|f(f+g)^{p-1}\|_1 + \|g(f+g)^{p-1}\|_1
\\ & \leq \|f^p\|_1^{1/p}\|(f+g)^{q(p-1)}\|_1^{1/q} + \|g^p\|_1^{1/p}\|(f+g)^{q(p-1)}\|_1^{1/q}
\\ & = (\|f^p\|_1^{1/p} + \|g^p\|_1^{1/p})\|(f+g)^p\|_1^{(p-1)/p}.
\end{align*}
The conclusion then follows.
\end{proof}

While one can prove the following result using the variation of H\"older's inequality for not necessarily commuting functions, it is simpler to get it as a direct result of the previous theorem.

\begin{corollary}
Suppose $\mathcal H=\mathbb C^n$ and $p\geq 1$. If $f,g\in \mathcal L_\mathcal H(X,\nu)$ such that $\|f(x)\|^pI_n,\|g(x)\|^pI_n \in  L^1_\mathcal H(X,\nu)$ and $f,g\geq 0$ then $(f+g)^p \in L^1_\mathcal H(X,\nu)$ and 
\[
\|(f+g)^p\|_1^{1/p} \leq \big\|\|f(x)\|^pI_n\big\|_1^{1/p} + \big\|\|g(x)\|^pI_n\big\|_1^{1/p}
\]
\end{corollary}
\begin{proof}
Since $f+g \geq 0$ then by functional calculus
\begin{align*}
(f(x)+g(x))^p & \leq \|(f(x)+g(x))^p\|I_n
\\ & = \|f(x) + g(x)\|^pI_n
\\ & \leq (\|f(x)\|+\|g(x)\|)^pI_n
\\ & = (\|f(x)\|I_n + \|g(x)\|I_n)^{p}.
\end{align*}
Thus, by comparison and the previous theorem, because multiples of the identity commute,
\begin{align*}
\|(f+g)^p\|_1^{1/p} & \leq \big\|(\|f(x)\|I_n + \|g(x)\|I_n)^p \big\|_1^{1/p}
\\ & \leq  \big\|\|f(x)\|^pI_n\big\|_1^{1/p} + \big\|\|g(x)\|^pI_n\big\|_1^{1/p}
\end{align*}
\end{proof}

These do not generalize well into infinite dimensions since Ando's matrix Young's inequality only really makes sense for compact operators \cite{EFZ}. As well, the above two results do not extend from positive to arbitrary functions. The crux of the problem is that multiplication does not work well with positive operators and the Loewner order (for $P,Q$ self-adjoint, $P\leq Q$ if and only if $Q-P$ is a positive operator).

A possibility to fix this is to consider tensor products, forcing commutativity. Suppose $\mu$ is a finite, positive measure on $X$ and $\nu_1$ and $\nu_2$ are POVMs into $\mathcal B(\mathcal H_1)$ and $\mathcal B(\mathcal H_2)$, respectively, such that $\nu_1,\nu_2 \ll_{ac} \mu$. One can define a POVM 
\[
\nu_1 \otimes_\mu \nu_2(E) = \int_E \left( \frac{d\nu_1}{d\mu} \otimes \frac{d\nu_2}{d\mu} \right) d\mu I_{\mathcal H_1 \otimes \mathcal H_2}
\]
provided that $\nu_1 \otimes_\mu \nu_2(X) \in \mathcal B(\mathcal H_1\otimes\mathcal H_2)$.

The choice of $\mu$ makes a difference: if $\gamma$ is a finite, positive measure such that $\mu <<_{ac} \gamma$ then 
\begin{align*} 
\nu_1 \otimes_\gamma \nu_2(E) & = \int_E \left( \frac{d\nu_1}{d\gamma} \otimes \frac{d\nu_2}{d\gamma} \right) d\gamma I_{\mathcal H_1 \otimes \mathcal H_2}
\\ & = \int_E \left( \frac{d\nu_1}{d\mu}\frac{d\mu}{d\gamma}I_{\mathcal H_1} \otimes \frac{d\nu_2}{d\mu}\frac{d\mu}{d\gamma}I_{\mathcal H_2} \right)d\gamma I_{\mathcal H_1\otimes \mathcal H_2}
\\ & = \int_E \frac{d\mu}{d\gamma}\left( \frac{d\nu_1}{d\mu} \otimes \frac{d\nu_2}{d\mu} \right) d\gamma I_{\mathcal H_1 \otimes \mathcal H_2}
\end{align*}
which in general is not equal to $\nu_1 \otimes_\mu \nu_2(E)$.

An alternative definition would be to take the square root of each Radon-Nikod\'ym derivative. While this would give a tensor product of measures that is the same no matter the choice of $\mu$, the downside is that one loses the structure discussed below.

Recall that if $s_1\in \mathcal S(\mathcal H_1)$ and $s_2\in \mathcal S(\mathcal H_2)$ then $s_1\otimes s_2 \in \mathcal S(\mathcal H_1\otimes \mathcal H_2)$, called a {\em product state}. In this way one can see that 
\[
\mathcal S(\mathcal H_1\otimes \mathcal H_2)_{sep} := \overline{\operatorname{conv}\{s_1\otimes s_2 : s_1\in \mathcal S(\mathcal H_1), s_2\in\mathcal S(\mathcal H_2)\}} \subset \mathcal S(\mathcal H_1\otimes \mathcal H_2).
\]
These are called the {\em separable states} and it is a well known fact that not all states are separable, calling such a non-separable state an {\em entangled state}.

Define a new seminorm using only the separable states
\[
\|f\otimes g\|_{p,sep,\nu_1 \otimes_{\mu} \nu_2} := \sup\{ \| (f\otimes g)_s\|_{p,\mu} : s \in\mathcal S(\mathcal H_1 \otimes \mathcal H_2)_{sep} \}
\]
and in the same way define a decomposable version 
\begin{align*}
& \|f\otimes g\|_{p,sep,dec} :=  \inf\bigg\{ \max\{ \|S_1\|_{p,sep, \nu_1\otimes_\mu \nu_2}, \|S_2\|_{p,sep, \nu_1\otimes_\mu \nu_2} \} :
\\  & \quad \quad \quad   S_1,S_2\in L^p_{\mathcal H_1\otimes\mathcal H_2}(X,\nu_1\otimes_\mu \nu_2), S_1,S_2\geq 0, \left[\begin{matrix} S_1(x) & f(x) \\ f(x)^* & S_2(x) \end{matrix}\right] \geq 0 \ \textrm{a.e.} \bigg\}.
\end{align*}
Using these we can establish a version of H\"older's inequality.

\begin{theorem}
If $f\in L^p_{\mathcal H}(X,\nu)$ and $g\in L^q_{\mathcal H}(X,\nu)$ for $\frac{1}{p} + \frac{1}{q} =1$ then
\[
\|f\otimes g\|_{1,sep,\nu \otimes_{\nu_\rho} \nu} \leq \|f\|_{p}\|g\|_{q} 
\]
and
\[
\|f\otimes g\|_{1,sep, dec} \leq \|f\|_{p,dec}\|g\|_{q,dec}.
\]
\end{theorem}
\begin{proof}
First, if $f$ and $g$ are positive and $s_1, s_2\in \mathcal S(\mathcal H)$ then
\begin{align*}
\|(f\otimes g)_{s_1 \otimes s_2}\|_1 
& = \int_X \tr \left(\left(s_1\frac{d\nu}{d\nu_\rho}^{1/2}f\frac{d\nu}{d\nu_\rho}^{1/2}\right)\otimes\left(s_2\frac{d\nu}{d\nu_\rho}^{1/2}g\frac{d\nu}{d\nu_\rho}^{1/2}\right) \right) d\nu_\rho
\\ & = \int_X \tr \left(s_1\frac{d\nu}{d\nu_\rho}^{1/2}f\frac{d\nu}{d\nu_\rho}^{1/2}\right)\tr \left(s_2\frac{d\nu}{d\nu_\rho}^{1/2}g\frac{d\nu}{d\nu_\rho}^{1/2}\right) d\nu_\rho
\\ & = \|f_{s_1}g_{s_2}\|_1
\\ & \leq \|f_{s_1}\|_p\|g_{s_2}\|_q
\\ & \leq \|f\|_{p}\|g\|_{q},
\end{align*}
using the classic H\"older's inequality.
Now suppose that $s_{i,j} \in \mathcal S(\mathcal H), i=1,2, 1\leq j\leq n$ and $\lambda_j\geq 0$ such that $\sum_{j=1}^n \lambda_j = 1$. If $s = \sum_{j=1}^n \lambda_j(s_{1,j}\otimes s_{2,j})$ then $s\in \mathcal S(\mathcal H\otimes\mathcal H)_{sep}$ and 
\begin{align*}
\|(f\otimes g)_s\|_1 
& = \int_X (f\otimes g)_s d\nu_\rho
\\ & = \sum_{j=1}^n \lambda_j \|(f\otimes g)_{s_{1,j}\otimes s_{2,j}}\|_1
\\ & \leq \sum_{j=1}^n \lambda_j \|f\|_{p}\|g\|_{q}
\\ & = \|f\|_{p}\|g\|_{q}.
\end{align*}
It is straightforward that if $s_k \rightarrow s$ in $\mathcal S(\mathcal H)$ then $f_{s_k} \rightarrow f_s$ in $L^1(X,\nu_\rho)$. Hence, for all $s\in \mathcal S(\mathcal H\otimes \mathcal H)_{sep}$
\[
\|(f\otimes g)_s\|_1 \leq \|f\|_{p}\|g\|_{q}.
\]

Secondly, suppose $(f_1,f_2,f_3,f_4)\in Pos_f$ and $(g_1,g_2,g_3,g_4)\in Pos_g$. As elsewhere, we have
\begin{align*}
\|f\otimes g\|_{1,sep,\nu\otimes_{\nu_\rho} \nu} &\leq \|(f_1 + f_2 +f_3 + f_4)\otimes(g_1+g_2+g_3+g_4)\|_{1,sep,\nu\otimes_{\nu_\rho} \nu}
\\ & = \sup_{s\in\mathcal S(\mathcal H\otimes \mathcal H)_{sep}} \|((f_1 + f_2 +f_3 + f_4)\otimes(g_1+g_2+g_3+g_4))_s\|_1
\\ & \leq \|f_1+f_2+f_3+f_4\|_{p}\|g_1+g_2+g_3+g_4\|_{q}.
\end{align*}
Taking the infimum over all possible combinations gives the first desired result.

Lastly, for general $f$ and $g$ suppose that $S_1,S_2\geq 0$ in $L^p_\mathcal H(X,\nu)$, $T_1,T_2\geq 0$ in $L^q_\mathcal H(X,\nu)$ such that
\[
\left[ \begin{matrix} S_1 & f \\ f^* & S_2 \end{matrix}\right] \geq 0 \quad \textrm{and} \quad \left[ \begin{matrix} T_1 & g \\ g^* & T_2 \end{matrix}\right] \geq 0.
\]
This implies that 
\[
\left[ \begin{matrix} S_1\otimes T_1 & f\otimes g \\ (f\otimes g)^* & S_2\otimes T_2 \end{matrix}\right] \geq 0.
\]
Thus, 
\begin{align*}
\|f\otimes g\|_{1,sep,dec} & \leq \max\{\|S_1\otimes T_1\|_{1, sep, \nu \otimes_{\nu_\rho} \nu}, \|S_2\otimes T_2\|_{1, sep, \nu \otimes_{\nu_\rho} \nu}\}
\\ & \leq \max\{ \|S_1\|_{p}\|T_1\|_{q},\ \|S_2\|_{p}\|T_2\|_{q}\}
\\ & \leq \max\{\|S_1\|_{p}, \|S_2\|_{p}\}\max\{\|T_1\|_{q}, \|T_2\|_{q}\}
\end{align*}
and taking infimums over all such $S_1,S_2,T_1,T_2$ leads to the second desired inequality.
\end{proof}

In \cite{PR2} it was shown that $L^\infty_\mathcal H(X,\nu)$ functions are not bounded multipliers of $L^1_\mathcal H(X,\nu)$ in general. Using the tensor product gets around this under invertibility and boundedness conditions on the Radon-Nikod\'ym derivative.

\begin{theorem}
Suppose $\frac{d\nu}{d\nu_\rho}\in L^\infty_\mathcal H(X,\nu)$ and $\frac{d\nu}{d\nu_\rho}(x)$ is invertible a.e.
If $f\in L^1_\mathcal H(X,\nu)$ and $g\in L^\infty_\mathcal H(X,\nu)$ then $f\otimes g\in L^1_\mathcal H(X,\nu\otimes_{\nu_\rho} \nu)$ such that
\[
\|f\otimes g\|_{1} \leq 2\left\|\frac{d\nu}{d\nu_\rho}\right\|_\infty|\|f\|_1\|g\|_\infty
\]
and
\[
\|f\otimes g\|_{1,dec} \leq \left\|\frac{d\nu}{d\nu_\rho}\right\|_\infty|\|f\|_{1,dec}\|g\|_\infty.
\]
\end{theorem}
\begin{proof}
This uses calculations very similar to \cite[Proposition 3.13]{PR2}.
Let $(f_1,f_2, f_3,f_4)\in Pos_f$. If $g$ is self-adjoint then its positive and negative parts are also essentially bounded. Moreover,
\begin{align*}
\|f\otimes g\|_1 &  = \left\| (f_1-f_2 + i(f_3 - f_4))\otimes(g_+ - g_-) \right\|_1
\\ & = \bigg\| f_1\otimes g_+ + f_2\otimes g_- - f_1 \otimes g_- - f_2\otimes g_+ 
\\ & \quad \quad \quad + i(f_3\otimes g_+ + f_4\otimes g_- - f_3\otimes g_- - f_4\otimes g_+) \bigg\|_1
\\ & \leq \bigg\| \int_X f_1\otimes g_+ + f_2\otimes g_- + f_1 \otimes g_- + f_2\otimes g_+ 
\\ & \quad \quad \quad  + f_3\otimes g_+ + f_4\otimes g_- + f_3\otimes g_- + f_4\otimes g_+ \: d\nu\otimes_{\nu_\rho} \nu \bigg\|
\\ & = \left\| \int_X (f_1+f_2+f_3+f_4)\otimes(g_+ + g_-)\: d\nu\otimes_{\nu_\rho} \nu \right\|
\\ & = \left\| \int_X (f_1+f_2+f_3+f_4)\otimes |g|\: d\nu\otimes_{\nu_\rho} \nu \right\|.
\end{align*}
By \cite[Lemma 3.8]{PR2} this last equation is equal to
\begin{align*}
&\left\| \int_X \frac{d\nu}{d\nu_\rho}^{1/2}(f_1+f_2+f_3+f_4)\frac{d\nu}{d\nu_\rho}^{1/2}\otimes \frac{d\nu}{d\nu_\rho}^{1/2}|g|\frac{d\nu}{d\nu_\rho}^{1/2}\: d\nu_\rho I_{\mathcal H\otimes\mathcal H} \right\|
\\ &\leq \left\| \int_X \frac{d\nu}{d\nu_\rho}^{1/2}(f_1+f_2+f_3+f_4)\frac{d\nu}{d\nu_\rho}^{1/2}\otimes\left\|\frac{d\nu}{d\nu_\rho}\right\|_\infty\|g\|_\infty I_\mathcal H\: d\nu_\rho I_{\mathcal H\otimes\mathcal H}\right\|
\\  &= \left\| \int_X \frac{d\nu}{d\nu_\rho}^{1/2}(f_1+f_2+f_3+f_4)\frac{d\nu}{d\nu_\rho}^{1/2}\: d\nu_\rho I_\mathcal H\right\|\left\|\frac{d\nu}{d\nu_\rho}\right\|_\infty\|g\|_\infty
\\ & = \left\| \int_X f_1+f_2+f_3+f_4 \:d\nu \right\|\left\|\frac{d\nu}{d\nu_\rho}\right\|_\infty\|g\|_\infty
\end{align*}
By taking the infimum over $Pos_f$ we get that $\|f\otimes g\|_1 \leq \left\|\frac{d\nu}{d\nu_\rho}\right\|_\infty\|f\|_1\|g\|_\infty$. 

For arbitrary $g$ the triangle inequality gives
\begin{align*}
\|f\otimes g\|_1 & \leq \|f\otimes Re\: g\|_1 + \|f\otimes Im\: g\|_1
\\ & \leq \left\|\frac{d\nu}{d\nu_\rho}\right\|_\infty\|f\|_1\|Re\: g\|_\infty + \left\|\frac{d\nu}{d\nu_\rho}\right\|_\infty\|f\|_1\|Im\: g\|_\infty
\\ & \leq 2\left\|\frac{d\nu}{d\nu_\rho}\right\|_\infty\|f\|_1\|g\|_\infty.
\end{align*}

For the decomposable norm suppose $S_1,S_2\in L^1_\mathcal H(X,\nu)$ with $S_1,S_2\geq 0$ and
$\left[\begin{matrix}S_1 & f \\ f^* & S_2\end{matrix} \right] \geq 0$. Copying the ideas of Lemma \ref{lemma:absolutevalue} we get that 
\[
\left[\begin{matrix}S_1\otimes |g^*| & f\otimes g \\ f^*\otimes g^* & S_2\otimes |g|\end{matrix} \right] \geq 0.
\]
This implies that by using the previous arguments
\begin{align*}
\|f\otimes g\|_{1,dec} & \leq \max\{ \|S_1\otimes |g^*|\|_1, \|S_2\otimes |g|\|_1\}
\\ & \leq \max\left\{\|S_1\|_1\left\|\frac{d\nu}{d\nu_\rho}\right\|_\infty\||g^*|\|_\infty, \|S_2\|_1\left\|\frac{d\nu}{d\nu_\rho}\right\|_\infty\||g|\|_\infty\right\}
\\ & = \max\{\|S_1\|_1,\|S_2\|_1\}\left\|\frac{d\nu}{d\nu_\rho}\right\|_\infty\|g\|_\infty.
\end{align*}
Taking the infimum over all possible $S_1,S_2$ yields the desired result.
\end{proof}

\vskip 12 pt

For comparison, we end this section with a variation of H\"older's inequality with the Schatten type norm. However, it is a tad unsatisfying as one cannot make the left-hand side behave itself.

\begin{lemma}\label{lemma:holderpp}
If $f,g \in \mathcal L_\mathcal H(X,\nu)$ and $\frac{1}{p} + \frac{1}{q} = 1$ then
\[
\sup_{s\in\mathcal S(\mathcal H)} \left\|\left\| s^{1/2}\frac{d\nu}{d\nu_\rho}^{1/2}f\frac{d\nu}{d\nu_\rho}^{1/2}s\frac{d\nu}{d\nu_\rho}^{1/2}g\frac{d\nu}{d\nu_\rho}^{1/2}s^{1/2} \right\|_{S^1}\right\|_{L^1} \leq \|f\|_{S^p,L^p}\|g\|_{S^q,L^q}
\]
\end{lemma}
\begin{proof}
First label $f' = s^{1/2}\frac{d\nu}{d\nu_\rho}^{1/2}f\frac{d\nu}{d\nu_\rho}^{1/2}s^{1/2}$ and $g' = s^{1/2}\frac{d\nu}{d\nu_\rho}^{1/2}g\frac{d\nu}{d\nu_\rho}^{1/2}s^{1/2}$.
By Young's inequality for the trace norm, where $A,B>0$,
\[
\frac{\tr(|f'(x)g'(x)|)}{AB} \leq \frac{\tr(|f'(x)|^p)}{A^p p} + \frac{\tr(|g'(x)|^q)}{B^q q},
\]
for every $x\in X$.
Hence, for $A=\left(\int_X \tr(|f'|^p)d\nu_\rho\right)^{1/p}$ and $B= \left(\int_X \tr(|g'|^q)d\nu_\rho\right)^{1/q}$ one has
\begin{align*}
\int_X \frac{\tr(|f'g'|)}{AB} d\nu_\rho & \leq \int_X \frac{\tr(|f'|^p)}{A^p p} + \frac{\tr(|g'|^q)}{B^q q} d\nu_\rho
\\ & = \frac{1}{p} + \frac{1}{q} = 1.
\end{align*}
Therefore, the conclusion follows.
\end{proof}

\subsection*{Acknowledgment}
The first author was supported by NSERC Discovery Grant 2019-05430. The second author was partially supported by MacEwan University USRI-Project grants. Both authors would like to thank the reviewer for their insightful comments and corrections.

\end{document}